\documentclass{article}
\usepackage{amsthm}
\usepackage{amsmath}
\usepackage{amssymb}
\usepackage[toc,page]{appendix}
\usepackage{graphicx}
\usepackage[round]{natbib}
\usepackage[final]{hyperref}
\usepackage{bbm}
\usepackage{float}
\usepackage{appendix}
\usepackage{enumitem}
\usepackage{pgf,tikz,pgfplots}
\usepackage{comment}
\usepackage{lscape}
\usepackage{color, colortbl}
\usepackage{adjustbox}

\graphicspath{ {Media/} }
\pgfplotsset{compat=1.15}
\theoremstyle{plain}  
\newtheorem{theorem}{Theorem}[section] 
\newtheorem{lemma}[theorem]{Lemma} 
\newtheorem{proposition}[theorem]{Proposition} 
\newtheorem{corollary}[theorem]{Corollary}

\theoremstyle{definition} 
\newtheorem{definition}[theorem]{Definition}
\newtheorem{example}[theorem]{Example}

\theoremstyle{remark}

\newcommand{\diff}{\,\mathrm{d}}
\newcommand{\E}{\mathbb{E}}
\newcommand{\var}{\operatorname{var}}

\newcommand{\ES}{\operatorname{ES}}

\newcommand{\R}{\mathbb{R}}

\newcommand{\one}{\mathbbm{1}}

\newcommand\restr[2]{{
  \left.\kern-\nulldelimiterspace 
  #1 
  \vphantom{\big|} 
  \right|_{#2} 
  }}

\begin{document}

\title{Isotonic regression for functionals of elicitation complexity greater than one}
\author{Anja M\"uhlemann and Johanna F.~Ziegel}
\maketitle
\begin{abstract}
We study the non-parametric isotonic regression problem for bivariate elicitable functionals that are given as an elicitable univariate functional and its Bayes risk. Prominent examples for functionals of this type are (mean, variance) and (Value-at-Risk, Expected Shortfall), where the latter pair consists of important risk measures in finance. We present our results for totally ordered covariates but extenstions to partial orders are given in the appendix.
\end{abstract}

\section{Introduction}
In isotonic regression the aim is to fit an increasing function $g_1$ to observations $(z_1, y_1), \dots, (z_n,y_n)$ such that a chosen loss function is minimized by $g_1$. The solution $g_1$ is then called a solution to the isotonic regression problem. If $g_1$ is supposed to model a conditional mean, then the loss function should be consistent for the mean in the sense of \citet[Definition 1]{Gneiting2011} with a prominent example being the squared error loss. More generally, if $g_1$ is a model for a conditional functional $T$, then the loss function  $L:\R \times \R \to \R$  should be chosen consistent for this functional $T$, that is, $\E_P L(t,Y) \leq \E_P L(x,Y)$ for all relevant probability distributions $P$, all $t \in T(P)$ and all $x \in \R$. 
Loss $L$ is called strictly consistent if the above inequality is strict for all $x \not\in T(P)$. This notion of consistency is a property of the functional $T$ and the loss function $L$ and should not to be confused with consistency of an estimator. Strict consistency of $L$ ensures that a correctly specified model minimizes the expected loss at the population level. 

If a functional $T$, that is, a map on a certain class of probability distributions, has a strictly consistent loss function it is called elicitable. We say that the loss function elicits $T$.  Elicitability is important for forecast comparison \citep{Gneiting2011}, and yields natural estimation procedures. Unfortunately, some ubiquitous functionals are not elicitable with prominent examples given by the variance ($\var$) and expected shortfall ($\ES_\alpha$), the latter being an important risk measure in finance and insurance. However, although $\ES_\alpha$ is not elicitable, it is jointly elicitable together with the $\alpha$-quantile ($q_\alpha$); see \cite{Fissler2016} and Example \ref{ex:2.2}.
Similarly, while $\var$ itself is not elicitable, it is jointly with the mean $(\E)$.
This means that both $\ES_\alpha$ and $\var$ are $2$-elicitable, that is, they can both be obtained as a function of a 2-dimensional elicitable functional. In a nutshell, the \emph{elicitation complexity} of a functional is the minimal number $k$ of dimensions needed for the functional to be $k$-elicitable. Since both $\ES_\alpha$ and $\var$ are not elicitable themselves but $2$-elicitable their elicitation complexity equals 2 \citep[Corollary 1 and 3]{Frongillo2018}.

Isotonic regression for one-dimensional elicitable functionals is well-under\-stood \citep{Barlow1972}. An interesting aspect is its robustness with respect to the choice of the consistent loss function in the minimization problem. In other words, no matter which strictly consistent loss function we choose for the functional $T$, we will obtain the same isotonic solution \citep{Brummer2013,Jordan2019}. This is in stark contrast to estimation in parametric regression models. In finite samples or for misspecified models, the choice of the consistent loss function may lead to miscellaneous estimates \citep{Patton2019}.

In this article, we investigate non-parametric regression for bivariate functionals $\underline{T}$ under isotonicity constraints. In particular, we show that simultaneous optimality with respect to an entire class of losses can rarely be achieved, and discuss how to find optimal solutions for specific choices of loss functions. The functionals we consider are of the form 
\[
\underline{T} = (T,\underline{L}),
\]
where $T$ is a one-dimensional elicitable functional with strictly consistent loss function $L$, and 
\begin{equation}\label{eq:BayesRisk}
\underline{L}(P):=\inf_{x_1\in\R} L(x_1,P)
\end{equation} 
with $L(x_1,P)=\int_{-\infty}^\infty L(x_1,y) \diff P(y)$ is the \emph{Bayes risk}. The example $\underline{T} = (\E,\var)$ arises by choosing $L(x,y) = (x-y)^2$, and the example $\underline{T}=(q_\alpha,\ES_\alpha)$ is obtained by choosing $L(x_,y) = (1/\alpha)\one\{ y\leq x\}(x-y)- x$, which is the piecewise linear loss known from quantile regression up to a function that only depends on $y$. Generally, \cite{Frongillo2018} show that $\underline{T}$ is always $2$-elicitable. Moreover, they also introduce a large class $\underline{\mathcal{L}}$ of loss functions ${L}(x_1,x_2,y)$ eliciting $\underline{T}$.

We show how the isotonic regression problem can be solved for $\underline{T}$. It turns out that the proposed canonical solution is generally not optimal with respect to all loss functions in $\underline{\mathcal{L}}$, but there is a fairly simple approach to check whether a given fit is simultaneously optimal. Furthermore, we show how the fit can be improved for a specific chosen loss function.  In a simulation experiment, we investigate how often simultaneously optimal fits occur for the functionals $(q_\alpha, \ES)$ and $(\E, \var)$ and investigate the fits for a specific choice of loss function.

The article is organized as follows. Section \ref{sec:mixture} introduces necessary preliminaries on consistent loss functions including a mixture representation for loss functions in $\underline{\mathcal{L}}$. In Section \ref{sec:isotonic}, the isotonic regression problem for total orders is formulated and a natural solution through sequential optimization is proposed. Then, we study the simultaneous optimality of the solution of the sequential optimization approach. Section \ref{sec:simulation} contains the numerical examples. In the Appendix, we show how our results can be generalized to partial orders.

\section{Preliminaries}\label{sec:mixture}

Following \citet{Jordan2019}, a function $V: \R \times \R \to \R$ is called an \emph{identification function} if $V(\cdot,y)$ is increasing and left-continuous for all $y\in\R$.
Then, for any probability measure $P$ on $\R$ with finite support, we define the functional T induced by an identification function $V$ as 
\begin{align*}
	T(P)= [T^-(P),T^+(P)] \subseteq [-\infty, \infty],
\end{align*}
where the lower and upper bounds are given by
\begin{align*}
	T^-(P)=\sup\{ x_1 : V(x_1,P) <0\} 
	\quad
	\text{and}
	\quad
	T^+(P)=\inf\{x_1:V(x_1,P)>0\},
\end{align*}
using the notation $V(x_1,P)=\int_{-\infty}^\infty V(x_1,y) \diff P(y)$. A broad class of functionals can be defined via their identification function, quantiles and expectiles, including the median and the mean, just being some of the most prominent examples. For other popular examples, see \cite{Jordan2019}. The examples of quantiles and expectiles already illustrate that the functional $T$ can take singleton-values as well as interval-values. 

Theorem 1 in \citet{Frongillo2018} states that if $L$ is a strictly consistent loss function for $T$ and $\underline{L}$ is the Bayes risk defined at \eqref{eq:BayesRisk}, then the loss
\begin{align}\label{eq:loss}
	\tilde{L}(x_1, x_2, y) = L'(x_1,y) + H(x_2) + h(x_2)(L(x_1,y)-x_2)
\end{align}
elicits $\underline{T}=(T,\underline{L})$, where $h:\R \to \R$ is any positive strictly decreasing function, $H(r) = \int_{0}^r h(x) \, \diff x$, and $L'$ is any consistent loss function for $T$ (possibly different from $L$ or even equal to zero). If $h$ is merely decreasing, then $\tilde{L}$ is still a consistent loss function.

\cite{Ehm2016} showed that for expectiles and quantiles any consistent loss function $L'$ can be written as
\begin{equation}\label{eq:mix}
	L'(x_1,y)=\int_\R S_{\eta,1}(x_1,y) \diff H_1(\eta),
\end{equation}
for certain elementary (quantile or expectile) losses $S_{\eta,1}$ and a measure $H_1$ on $\R$ depending on $L'$. In fact, such mixtures always yield a large class $\mathcal{L}$ of consistent scoring functions for $T$ if it is identifiable with identification function $V(x,y)$ \citep{Dawid2016,Ziegel2016a}. Then, the elementary losses are given by
\begin{align}\label{eq:1}
	S_{\eta,1}(x_1,y)=
	\left(\one\{\eta \leq x_1\}
	-\one\{\eta \leq y\}\right)V(\eta,y),
\end{align}
where $\eta \in \R$. Moreover, the elementary losses are themselves consistent for $T$. We define
\begin{align*}
	\mathcal{L} 	=\left\lbrace (x_1,y) \mapsto \int_\R S_{\eta,1}(x_1,y) \diff H_1(\eta):	H_1 \text{ is a positive measure on } \R \right\rbrace.
\end{align*}
Note that (strict) consistency of a loss function is not altered by adding functions in $y$ as long as they are integrable for all relevant probability measures $P$. Therefore, when speaking of characterizations of the class of (strictly) consistent loss functions this is always meant up to possible addition of a function in $y$. 

If a loss function is given as a mixture of elementary losses as in \eqref{eq:mix}, this may be useful when minimizing the expected loss (over some set of parameters, for example); see details for the isotonic regression problem in Section \ref{sec:isotonic}. Using Fubini's theorem, one can see that we can look for minimizers of the expected elementary losses and hope that these minimizers all agree, that is, there is a simultaneous minimizer for all parameters $\eta$. Then, this minimizer is automatically optimal for all scoring functions of the form \eqref{eq:mix}, independently of the measure $H_1$. Indeed, this approach is at the heart of the characterization of all simultaneously optimal solutions to the isotonic regression problem for one-dimensional functionals in \cite{Jordan2019}.

Using the same approach as used by \cite{Ziegel2020} to derive a mixture representation for the pair $(q_\alpha,\ES_\alpha)$, we derive a mixture representation for the loss functions for $\underline{T}$ of the form \eqref{eq:loss}.

\begin{lemma}\label{lem:mixture}
Let $L,L'\in{\mathcal{L}}$. Then, all consistent loss functions for $\underline{T} = (T,\underline{L})$ defined at \eqref{eq:loss} are of the form 
\begin{align} \label{eq:mixture}
	\tilde{L}(x_1,x_2,y)= \int S_{\eta,1}(x_1,y) \, \diff H_1(\eta)
	+ \int S_{\eta,2}(x_1,x_2,y) \, \diff H_2(\eta),
\end{align} 
where $H_1, H_2$ are measures on $\R$, $H_2$ is finite on intervals of the form $(-\infty,-x_2]$, $x_2\in \R$, and
\begin{align*}
	S_{\eta,1}(x_1, y) &= (\one\{\eta \leq x_1\} - \one\{\eta \leq y\}) 
	V(\eta,y)\\
	S_{\eta,2}(x_1,x_2,y) & = \one\{\eta \leq -x_2\} (L(x_1,y)+\eta) - \one\{\eta \le 0\}\eta.
\end{align*}
Conversely, any loss function of the form \eqref{eq:mixture} is consistent for $\underline{T} = (T,\underline{L})$. It is strictly consistent if $H_2$
puts positive mass on all open intervals.
\end{lemma}
\begin{proof}
The consistency follows directly from Theorem 1 in \cite{Frongillo2018}. 
Recall that $h$ is decreasing and nonnegative and $H(r)=\int_{0}^r h(x) \diff x$. To see that the loss functions in \eqref{eq:loss} with loss $L' \in \mathcal{L}$ can be written as in \eqref{eq:mixture}, define $A: =\lim_{x \to \infty} h(x) \geq 0$. Since $h \geq 0$, we can define the measure $H_2$ by $H_2((-\infty, t]) = h(-t)-A \geq 0$ for all $t \in \R$. Without loss of generality we can assume $h$ satisfies $\lim_{x \to \infty}h(x) = 0$. Indeed, we can define $\underline{h}=h-A$ then $H$ becomes $\underline{H}(x)=H(x)-xA$ and $\tilde{L}(x_1,x_2,y)=\underline{\tilde{L}}(x_1,x_2,y)+AL(x_1,y)$. Then, $\underline{\tilde{L}}(x_1,x_2,y)= L'(x_1,y)-AL(x_1,y) + H(x_2) + h(x_2)(L(x_1,y)-x_2)$. Thus, adding constants to $h$ corresponds to modifying the loss function $L'$. Moreover, since $L, L' \in \mathcal{L}$ we have that $L'+AL \in \mathcal{L}.$ Hence, we can assume that $A=0$, then $H_2((-\infty,x]) = h(-x)$ for all $x \in \mathbb{R}$ and
\begin{align*}
	h(x_2) = \int_{-\infty}^{-x} \diff H_2(\eta).
\end{align*}
Then we have
\begin{align*}
	 \int S_{\eta,2}(x_1,x_2,y) \, \diff H_2(\eta) 
	 =&  L(x_1,y) h(x_2)
	 -  \int_{-x_2}^0 \eta \diff H_2(\eta).
\end{align*}
Integration by parts yields
\begin{align*}
	 \int S_{\eta,2}(x_1,x_2,y) \, \diff H_2(\eta) 
	 =&   L(x_1,y)h(x_2)
	 -  x_2 h(x_2) +H(x_2).
\end{align*}
Restricting the choice of $L'$ to $\mathcal{L}$ ensures the existence of the mixture representation for $L'(x_1,y)$.
\end{proof}

The following two examples discuss the mixture representations for the pairs $(q_\alpha,\ES_\alpha)$ and $(\E,\var)$ in more detail.

\begin{example}\label{ex:2.2}
As mentioned in the introduction, a popular but non-elicitable risk measure is expected shortfall. In this article we adopt the sign convention used by \cite{Frongillo2018} which is different from \citet{Fissler2016,Ziegel2020}.

For a given level $\alpha \in (0,1)$ the loss function 
\[
L(x_1,y)=\frac{1}{\alpha}\one\{ y\leq x_1\}(x_1-y)- x_1
\]
elicits the 
$\alpha$-quantile $q_\alpha(P)$.
The expected shortfall $\ES_\alpha$ is the corresponding Bayes risk, that is, 
\[
\ES_\alpha(P) = \inf_{x_1\in\R} L(x_1,P).
\]
The elementary loss functions of Lemma \ref{lem:mixture} are given by
%
\begin{align*}
	S_{\eta,1}(x_1, y) 
	&= (\one\{\eta \leq x_1\} - \one\{\eta \leq y\}) 
	(\one\{\eta > y \}-\alpha)\\
	S_{\eta,2}(x_1,x_2,y) 
	&= \one\{\eta \leq -x_2\} 
	\left(\tfrac{1}{\alpha}\one\{ y\leq x_1\}(x_1-y)- (x_1-\eta)\right) - \one\{\eta \le 0\}\eta.
\end{align*}
In fact, all loss functions consistent for the pair $(q_\alpha,\ES_\alpha)$ are of the form \eqref{eq:loss}, or equivalently, \eqref{eq:mixture}; see \citet{Ziegel2020}.
Due to the different sign conventions mentioned previously, the mixture representation in \citet{Ziegel2020} corresponds to $L(x_1,-x_2,y)$ (up to normalization).
\end{example}

\begin{example}
The squared loss $L(x_1,y)=(x_1-y)^2$ elicits the expectation $\E(P)$.
The corresponding Bayes risk is the variance $\var(P)$.
Thus, the pair $(\E,\var)$ is elicitable.
The elementary loss functions of Lemma \ref{lem:mixture} are given by
\begin{align*}
	S_{\eta,1}(x_1, y) 
	&= (\one\{\eta \leq x_1\} - \one\{\eta \leq y\}) 
	(\eta-y)\\
	S_{\eta,2}(x_1,x_2,y) 
	&= \one\{\eta \leq -x_2\} 
	\left((x_1-y)^2+\eta)\right)- \one\{\eta \le 0\}\eta.
\end{align*}
In contrast to the pair $(q_\alpha,\ES_\alpha)$ not all consistent loss functions for $(\E,\var)$ are of this form; see \citet[Section 3.1]{Frongillo2018}.
\end{example}

\section{Isotonic regression}\label{sec:isotonic}
\subsection{General results}\label{sec:general}
Suppose we have pairs of observations $(z_1,y_1), \dots, (z_n, y_n)$, where $y_1, \dots, y_n$ are real-valued, the covariates $z_1, \dots, z_n$ are equipped with a total order, and $z_1<z_2<\dots<z_n$.
Repeated observations can easily be accommodated; see Remark 3.1 in \citet{Jordan2019}. We aim to fit a function $\hat{g}=(\hat{g}_1, \hat{g}_2): \{z_1, \dots, z_n\}^2  \to \R^2$ to these observations, such that $g_1$ is isotonic and models the conditional functional $T$ given the covariates $z_i$, and $g_2$ is antitonic and models the conditional Bayes risk $\underline{L}$ given at \eqref{eq:BayesRisk} given the covariates $z_i$ for some consistent loss function $L\in \mathcal{L}$.
That is, if $z_i \leq z_j$ then $\hat{g}_1(z_i) \leq \hat{g}_1(z_j)$ and $\hat{g}_2(z_i) \geq \hat{g}_2(z_j)$, respectively.
Considering the pair $(q_\alpha,\ES_\alpha)$ for example, one would be interested in an isotonic $\hat{g}_1$ and an antitonic $\hat{g}_2$ since $q_\alpha(Y_1) \leq q_\alpha(Y_2)$ and $\ES_\alpha(Y_1) \geq \ES_\alpha(Y_2)$ whenever $Y_1 \leq Y_2$ almost surely. Keeping this leading example in mind, we focus on the case that $g_1$ is isotonic, or increasing, and $g_2$ is decreasing, or antitonic. Adaptations of the results, where $g_1$ is desired to be decreasing or $g_2$ to be increasing are straight forward.

Following the literature on loss functions for expected shortfall, we first consider loss functions of the form \eqref{eq:loss} with $L' = 0$ \citep{Nolde2017,PattonETAL2019}. When studying simultaneous optimality of solutions in Section \ref{sec:simultaneously}, we also consider $L' \not=0$.
Let $h:\mathbb{R} \to (0,\infty)$ be decreasing with $\lim_{x \to \infty}h(x) = 0$, $H(r) = \int_0^r h(x)\diff x$. The goal is to minimize
\begin{align}
\label{eq:ISOREG}
\sum_{i=1}^n \tilde{L}(g_1(z_i),g_2(z_i),y_i) = \sum_{i=1}^n \left(H(g_2(z_i)) + h(g_2(z_i))(L(g_1(z_i),y_i) - g_2(z_i))\right)
\end{align}
over all functions $g = (g_1,g_2):\{z_1, \dots, z_n\}^2  \to \R^2$ such that $g_1$ is increasing and $g_2$ is decreasing. Keeping either $g_1$ or $g_2$ fixed, we can directly give an optimal solution with respect to the other component.

\begin{proposition}\label{prop:fixgi}
\begin{enumerate}
\item[(a)] Let $g_1:\{z_1,\dots,z_n\} \to \mathbb{R}$ be given. Then, the optimal antitonic solution $\hat{g_2}$ of \eqref{eq:ISOREG} with $g_1$ fixed is given by
\[
\hat{g}_2(z_\ell) = -\min_{j \geq \ell} \max_{i \leq j} -\E(\bar{P}_{i:j}) = -\max_{i \leq \ell} \min_{j\geq i} - \E(\bar{P}_{i:j}),\quad \ell = 1,\dots,n,
\]
where $\bar{P}_{i:j}$ is the empirical distribution of $L(g_1(z_i),y_i),\dots,L(g_1(z_j),y_j)$.
\item[(b)] Let $g_2:\{z_1,\dots,z_n\} \to \mathbb{R}$ be given. Then, any optimal isotonic solution $\hat{g}_1$ of \eqref{eq:ISOREG} with $g_2$ fixed satisfies
\[
\min_{j \geq \ell} \, \max_{i \leq j} T^-(P^w_{i:j}) \leq \hat{g}_1(z_\ell) \leq \max_{i \leq \ell} \, \min_{j \geq i} T^+(P^w_{i:j}),
\]
\end{enumerate}
where $P^w_{i:j}$ is the weighted empirical distribution of $y_i,\dots,y_j$ with weights proportional to $h(g_2(z_i)),\dots,h(g_2(z_j))$.
\end{proposition}
\begin{proof}
\begin{enumerate}

\item[(a)]
Notice that for fixed $g_1$, the loss function \eqref{eq:ISOREG} is a Bregman loss function. Moreover, $\hat{g}_2$ is isotonic if and only if $-\hat{g}_2$ is antitonic. Thus, we can solve the classical isotonic regression problem as in \cite{Jordan2019} for $-\hat{g}_2$ to obtain the optimal antitonic $\hat{g}_2$.  

\item[(b)] Minimizing \eqref{eq:ISOREG} for fixed $g_2$ is equivalent to minimizing
\begin{align*}
	\sum_{i=1}^n h(g_2(z_i))L(g_1(z_i),y_i).
\end{align*}
Using the same reasoning as in Remark 3.1 in \cite{Jordan2019}, we have  $h(g_2(z_i))L(g_1(z_i),y_i)=L(g_1(z_i),P_{i:i}^{w})$. Finally, Proposition 3.6 in \cite{Jordan2019} yields the result.\qedhere
\end{enumerate}

\end{proof}

If $T$ is singleton-valued, Proposition \ref{prop:fixgi} yields the existence and a necessary conditions on any solution to \eqref{eq:ISOREG}.
\begin{corollary}\label{cor:singleton}
If $T$ is singleton-valued a solution $\hat{g_1}, \hat{g}_2$ to \eqref{eq:ISOREG} exists. In particular, we have
\[
\hat{g}_2(z_\ell) = -\min_{j \geq \ell} \max_{i \leq j} -\E(\bar{P}_{i:j}) = -\max_{i \leq \ell} \min_{j\geq i} - \E(\bar{P}_{i:j}),
\]
where $\bar{P}_{i:j}$ is the empirical distribution of $L(\hat{g}_1(z_i),y_i),\dots,L(\hat{g}_1(z_j),y_j)$, and
\[
\hat{g}_1(z_\ell)
=\min_{j \geq \ell} \, \max_{i \leq j} T(P^w_{i:j}) 
= \max_{i \leq \ell} \, \min_{j \geq i} T(P^w_{i:j}),
\]
where $P^w_{i:j}$ is the weighted empirical distribution of $y_i,\dots,y_j$ with weights proportional to $h(\hat{g}_2(z_i)),\dots,h(\hat{g}_2(z_j))$.
\end{corollary}

\begin{proof}
For all solutions that are given by a $\min$-$\max$-representation with respect to some functional $\tilde{T}$ there exists a partition $\mathcal{Q}$ of the index set with $g(z_\ell)=\tilde{T}(Q)$, $\ell \in Q$, $Q \in \mathcal{Q}$ \citep[Proposition 4.17]{Jordan2019}.
Since there exist only finitely many partitions of the index set $\{1, \dots,n\}$ there exist only finitely many possible solutions. Therefore, an optimal solution has to exist. In particular, $\hat{g}_1$ has to be the solution obtained from Proposition \ref{prop:fixgi} when $\hat{g}_2$ is treated as fixed and vice versa. Otherwise we could replace $\hat{g}_1$ by the solution obtained from Proposition \ref{prop:fixgi} to obtain a smaller loss. Similarly, we could replace $\hat{g}_2$ by the solution in Proposition \ref{prop:fixgi} to obtain a smaller loss.
\end{proof}

Furthermore, Proposition \ref{prop:fixgi} suggests an algorithm for finding minimizers of \eqref{eq:ISOREG}, which roughly consists of the following steps:
\begin{enumerate}
\item Take $g_2$ constant and find the optimal $\hat{g}_1^{(1)}$.
\item Find the optimal $\hat{g}_2^{(1)}$ given $\hat{g}_1^{(1)}$.
\item Find the optimal $\hat{g}_1^{(2)}$ given $\hat{g}_2^{(1)}$.
\item Iterate steps 2 and 3 until $\hat{g}_1^{(k)} = \hat{g}_1^{(k-1)}$.
\end{enumerate}

There is a problem with this algorithm if $T$ is interval-valued, since then, the solution in part (b) of Proposition \ref{prop:fixgi} is not unique. It turns out that it is best to choose the smallest possible solution corresponding to $T^-$, see Section \ref{sec:solutions} for details. 

\cite{Fissler2019} show that the expectation of consistent loss functions has no local minima. The optima in the isotonic regression case are more complex. But we believe that order sensitivity can be exploited to argue that the above algorithm can only converge to a global optimum. Numerical considerations where we perturbed the initial solutions to see whether they still converge to the same solution reinforced our suspicions that the algorithm does not converge to a saddle point. However, a rigorous mathematical proof for this conjecture is currently an open problem.

\subsection{Solution to the optimization problem}\label{sec:solutions}

In this somewhat technical section, will show that for fixed $g_2$ it is best to choose  
\begin{align}\label{eq:8}
	\begin{split}
	\hat{g}_1^-(z_\ell)&
	:= \min_{j \geq \ell} \max_{i \leq j} T^-(P_{i:j}^w)
	= \max_{i \leq \ell} \min_{j\geq i} T^-(P_{i:j}^w),
	\end{split}
\end{align}
where $P^w$ is the weighted empirical distribution with weights proportional to $h(g_2(z_i)),\dots, h(g_2(z_j))$, to minimize \eqref{eq:ISOREG}; see Propositions \ref{prop:main} and \ref{prop:minimizers}.

We denote $T^\lambda = \lambda T^- + (1-\lambda)T^+$, $\lambda \in [0,1]$, where $T^-$ and $T^+$ are the lower and upper bound of $T$, respectively. In \eqref{eq:8} the indices $\ell, i$ and $j$ are all elements of the index set $\{1, \dots, n\}$.
If we were to restrict $\ell, i$ and $j$ to be elements of the subset $\{1, \dots, m\}$, $m\leq n$, we would obtain an optimal solution on the subset $(z_1, y_1) ,\dots, (z_m,y_m)$ of the original data set. In the following, we denote an optimal solution on this subset by $\hat{g}_{1;1:m}$ and by $\restr{\hat{g}_1}{1:m}$ we denote the optimal solution on the original set restricted to $\{z_1, \dots, z_m\}$.

The following auxiliary result relates $\hat{g}_{1;1:m}$ to $\restr{\hat{g}_1}{1:m}$ in the case where $\hat{g}_1$ is given by a $\min$-$\max$-representation.
\begin{lemma}\label{lem:order}
Assume that
\begin{align*}
	\hat{g}_1(z_\ell)&
	:= \min_{j \geq \ell} \max_{i \leq j} T^\lambda(P_{i:j}^w)
	= \max_{i \leq \ell} \min_{j\geq i} T^\lambda(P_{i:j}^w)
\end{align*}
for some $\lambda \in [0,1]$. Then we have $\restr{\hat{g}_1}{1:m} \leq \hat{g}_{1;1:m}$.
\end{lemma}
\begin{proof}
Notice that
\begin{equation*}
	\hat{g}_{1;1:m}(z_\ell)
	= 	\mathop{\min_{j \geq \ell}}_{j\leq m} 
	\max_{i \leq j} T^\lambda(P_{i:j}^w)
	\geq \min_{j \geq \ell}
	\max_{i \leq j} T^\lambda(P_{i:j}^w)
	= \hat{g}_1 (z_\ell). \qedhere
\end{equation*}
\end{proof}

We recall some observations made in  \cite{Jordan2019}. For fixed weights, that is, for fixed $g_2$, we can minimize 
\begin{align*}
	\sum_{i=1}^n \one\{\eta \leq \hat{g}_1(z_i)\} V(\eta, P^w_{i:i}),
	\quad \text{for all }
	\eta \in \R
\end{align*}
to obtain a solution to \eqref{eq:ISOREG}. Because we want $\hat{g}_1$ to be isotonic, this means that for a given $\eta \in \R$ we have to find an index $\ell \in \{1. \dots, n+1\}$ that minimizes
\begin{align}\label{eq:index}
	\sum_{i=\ell}^n V(\eta, P^w_{i:i}).
\end{align}
The search for the optimal index $\ell$ needs to be conducted for every $\eta \in \R$. For $\eta \in \R$, we denote the set of indices minimizing \eqref{eq:index} by $I_{1:n}(\eta)$.

Recall that optimal solutions $\hat{g}_1$ are in one-to-one correspondence to increasing, left-continuous functions $\iota: \R \to \{1, \dots, n+1\}$ with $\iota(\eta)\in I_{1:n}(\eta)$, for all $\eta \in \R$, in the sense that 
\begin{align*}
	\inf \{ \eta : \iota(\eta)>\ell \} 
	= \hat{g}_1 (z_\ell)
	= \max\{\eta: \iota(\eta)\leq \ell\}.
\end{align*}
Thus, any solution to the isotonic regression problem yields a minimizing index $\iota(\eta)$ for every $\eta \in \R$.

The next result shows that if $\hat{g}_1$ is a solution to the isotonic regression problem \eqref{eq:ISOREG} with $\hat{g}_1(z_m)<\hat{g}_1(z_{m+1})$ then $\restr{\hat{g}_1}{1:m}$ is an optimal solution to the isotonic regression problem \eqref{eq:ISOREG} on the subsample $(z_1,y_1), \dots, (z_m,y_m)$.

\begin{lemma}\label{lem:index}
We have that $I_{1:n}(\eta) \cap \{1, \dots,m+1\} \subseteq I_{1:m}(\eta)$, where $ I_{1:m}(\eta)$ is the set of minimizing indices for the isotonic regression problem \eqref{eq:ISOREG} on the subsample $(z_1,y_1), \dots, (z_m,y_m)$.
\end{lemma}
\begin{proof}
Let $\ell \in I_{1:n}(\eta) \cap \{1, \dots,m\}$ for some $\eta \in \R$. Therefore, the function 
\begin{align*}
	t_\eta : \{1,\dots,n+1\}\to\R,
	x \mapsto \sum_{i=x}^n V(\eta,P_{i:i}^w)
\end{align*}
has a minimum at $\ell$. We can write
\begin{align*}
	\sum_{i=\ell}^n V(\eta,P_{i:i}^w)
	= \sum_{i=\ell}^m V(\eta,P_{i:i}^w)
	+ \sum_{i=m+1}^n V(\eta,P_{i:i}^w).
\end{align*}
Hence, $\restr{t_\eta}{1:m}$ has also a minimum at $\ell$ and thus $\ell \in I_{1:m}(\eta)$. If $t_\eta$ has a minimum at $\ell=m+1$ then
\begin{align*}
	t_\eta(x)-\sum_{i=m+1}^n V(\eta, P_{i:i}^w) \geq 0
\end{align*}
with equality for $x=m+1$. Thus, $I_{1:n}(\eta) \cap \{1, \dots,m+1\} \subseteq I_{1:m}(\eta)$.
\end{proof}

\begin{corollary}\label{cor:opt}
Let $\hat{g}_1$ be a solution \eqref{eq:ISOREG} with $\hat{g}_1(z_m)<\hat{g}_1(z_{m+1})$. Then we have that $\restr{\hat{g}_1}{1:m}$ is an optimal solution to \eqref{eq:ISOREG} on the subsample $(z_1,y_1), \dots, (z_m,y_m)$.
\end{corollary}

We now would like to show that for fixed weights the solution 
\begin{align*}
	\hat{g}_1^-(z_\ell)= \min_{j \geq \ell} \max_{i \leq j} T^-(P_{i:j}^w)
	= \max_{i \leq \ell} \min_{j\geq i} T^-(P_{i:j}^w)
\end{align*}
is most likely to minimize \eqref{eq:ISOREG}. An intuition behind this statement is obtained by combining Lemma \ref{lem:order} with Lemma 3.4 from \cite{Jordan2019}. The statement in Lemma \ref{lem:order} is equivalent to $\restr{\hat{g}_1^-}{(m+1):n}\geq \hat{g}_{1,(m+1):n}^-$. Lemma 3.4 of \cite{Jordan2019} on the other hand, implies that any optimal solution $\hat{g}_{1,(m+1):n}$ on $(z_{m+1}, y_{m+1}),\dots,(z_n,y_n)$ has to satisfy $\hat{g}_{1,(m+1):n}^- \leq \hat{g}_{1,(m+1):n} \leq \hat{g}_{1,(m+1):n}^+$. Thus, $\hat{g}_1^-$ has the highest chance to lie between those bounds. 

To prove this formally the order sensitivity of loss functions is needed. We recall the definition given in \cite{Steinwart2014}.
\begin{definition} Let $\mathcal{P}$ be a class of probability distributions.
A loss function $L:\R \times \R \to \R$ is said to be $\mathcal{P}$\emph{-order sensitive for T}, if the image of $T$ is an interval, and for all $P \in \mathcal{P}$ and all $t_1, t_2 \in \R$ with either $t_2<t_1\leq T^-(P)$ or $ T^+(P) \leq t_1<t_2$, we have $L(t_1,P) < L(t_2,P)$.
\end{definition} 

It follows directly from the definition that order sensitive loss functions are consistent. The reverse holds under weak regularity conditions on the functional; see \citet[Proposition 11]{Lambert2019}. 
The loss functions in class $\mathcal{L}$ are order-sensitive because they are defined via oriented identification function and a positive measure $H_1$ \cite[Theorem 7]{Steinwart2014}. Thus, the loss function $L$ in the following proposition is order sensitive.

\begin{proposition}\label{prop:main}
For fixed ${g}_2$, $\hat{g}_1^-$ given by \eqref{eq:8} and any increasing $\hat{g}_1$ we have
\begin{align*}
	\sum_{i=1}^n \tilde{L}(\hat{g}_1^-(z_i),g_2(z_i),y_i) 
	\leq 
	\sum_{i=1}^n \tilde{L}(\hat{g}_1(z_i),g_2(z_i),y_i).
\end{align*}
\end{proposition}
\begin{proof}
Note that for each $\hat{g}_1$ and we have a partition $\mathcal{Q}$ of the index set such that 
\begin{align*}
	\hat{g}_1(z_i)=\hat{g}_1(z_j)
	\quad 
	\text{for all }
	i,j \in Q, \, Q \in \mathcal{Q}.
\end{align*}
We let $Q_m$ denote the partition element corresponding to $\hat{g}_1$ containing $m$, and $Q_m^-$ denote the partition element corresponding to $\hat{g}_1^-$ containing $m$.

By Lemma \ref{lem:mixture}, it suffices to show that for all $\eta \in \R$
\begin{align*}
	\sum_{i=1}^n S_{\eta,2}(\hat{g}_1^-(z_i),g_2(z_i),y_i) 
	\leq 
	\sum_{i=1}^n S_{\eta,2}(\hat{g}_1(z_i),g_2(z_i),y_i).
\end{align*}
For the latter, it suffices to show that for all $m \leq n$
\begin{align}\label{eq:wts1_orig}
	\sum_{\ell=m}^n L(\hat{g}_1^-(z_\ell), y_\ell)
	\leq \sum_{\ell=m}^n L(\hat{g}_1(z_\ell), y_\ell).
\end{align}
This statement clearly holds if $\hat{g}_1$ has a jump in a some non-minimizing index, that is $\ell$ with $\ell \notin \cup_\eta I_{1:n}(\eta)$.  Thus, we can focus on $\hat{g}_1$ that solely jumps in $\ell$ with $\ell \in \cup_\eta I_{1:n}(\eta)$.  This implies that we have
\begin{align*}
	\sum_{\ell=1}^n L(\hat{g}_1^-(z_\ell), y_\ell)
	= \sum_{\ell=1}^n L(\hat{g}_1(z_\ell), y_\ell).
\end{align*}
In the following, we will prove the converse to \eqref{eq:wts1_orig}, that is, for all $m\leq n$ we have
\begin{align}\label{eq:wts1}
	\sum_{\ell=1}^m L(\hat{g}_1(z_\ell), y_\ell)
	\leq \sum_{\ell=1}^m L(\hat{g}_1^-(z_\ell), y_\ell).
\end{align}
If $m = \max Q_m$, it follows from Corollary \ref{cor:opt} that $\restr{\hat{g}_1}{1:m}$ is optimal on $(z_1, y_1),$ $\dots,$  $(z_m,y_m)$ and therefore \eqref{eq:wts1} holds. 

For $m \neq \max Q_m$, we distinguish two cases.\\
\noindent
\textbf{Case 1:} If $m = \max Q_m^-$, it follows from Lemma \ref{lem:order} and Proposition \ref{prop:fixgi} that
\begin{align*}
	\restr{\hat{g}_1^-}{1:m}
	= \hat{g}_{1 ; 1:m}^-
	\leq \restr{\hat{g}_1}{1:m}
	\leq \restr{\hat{g}_1^+}{1:m}
	\leq \hat{g}_{1; 1:m}^+.
\end{align*}
By Lemma \ref{lem:index} we have $I_{1:n}(\eta) \cap \{1, \dots, m+1\} \subseteq I_{1:m} (\eta)$ for all $\eta \in \R$. Hence, $\restr{\iota}{1:m}(\eta) \in I_{1:m}(\eta)$ for all $\eta \in \R$, where ${\iota}: \R \to \{1, \dots, n+1\}$ is the function imposing the score minimizing-indices corresponding to $\hat{g}_1$. Thus, Proposition 3.5 in \cite{Jordan2019} implies that $\restr{\hat{g}_1}{1:m}$ is a solution to the isotonic regression problem on $(z_1, y_1),\dots, (z_m,y_m)$.

\noindent
\textbf{Case 2:} Consider the case $m \neq \max Q_m$ and let $j=\max\left(\min Q_m,\min Q_m^-\right)$.
It follows from the previous considerations that $\hat{g}_1$ is optimal up to $j-1$ in the sense that it is a minimizer on $(z_1, y_1),\dots, (z_{j-1},y_{j-1})$.
We know that $	\restr{\hat{g}_1^-}{1:m}\leq \hat{g}_{1 ; 1:m}^-$ so if $\restr{\hat{g}_1}{1:m} \geq \hat{g}_{1; 1:m}^-$ we can conclude with the same reasoning as in case 1.

Otherwise, let $j_0 \geq j$ be the minimal index with $\hat{g}_{1;1:m}^-(z_{j_0}) > \restr{\hat{g}_1}{1:m}(z_{j_0})$. Clearly $j_0 \in Q_m$ and hence $\hat{g}_1$ is constant on $\{j,\dots, j_0\}$.  Moreover, if $j_0>j$ then for all $\ell \in \{1, \dots, j_0-1\}$ we have that
\begin{align*}
	\hat{g}_{1; 1:m}^-(z_\ell) 
	=\hat{g}_{1;1:(j_0-1)}^-(z_\ell) 	
	\leq \hat{g}_1(z_\ell)  
	\leq \restr{\hat{g}_{1}^+}{1:(j_0-1)}(z_\ell)
	\leq \hat{g}_{1; 1:(j_0-1)}^+(z_\ell),
\end{align*}
implying that $\hat{g}_1$ is in fact optimal up to $j_0-1$. Of course, if $j_0=j$, we already know that $\hat{g}_1$ is optimal up to $j_0-1$, since we know that $\hat{g}_1$ is optimal up to $j-1$ from our previous considerations. Thus, it remains to check what happens for $\ell \in \{j_0, \dots,m\}$. 

For $\ell \in \{j_0, \dots,m\}$ we have $\hat{g}_1^-(z_\ell) = c^- \leq c=\hat{g}_1(z_\ell)<\hat{g}_{1;1:m}^-(z_\ell)  $ for some constants $c^-$ and $c$.

Denote by $Q_{s;1:m}^-,\dots, Q_{r;1:m}^-$ the partition elements of $\hat{g}_{1;1:m}^-$ on $\{j_0, \dots,m\}$. Then, for $k \in \{s,\dots,r\}$ we have
\begin{align*}
	\sum_{\ell \in Q_{k;1:m}^-} L(\hat{g}_{1;1:m}^-, y_\ell)
	\leq
	\sum_{\ell \in Q_{k;1:m}^-} L(c, y_\ell)
	\leq
	\sum_{\ell \in Q_{k;1:m}^-} L(c^-, y_\ell)
\end{align*}
since $\hat{g}_{1;1:m}^-$ is constant each $Q_{k;1:m}^-$ and $L$ is order-sensitive. Therefore, \eqref{eq:wts1} is fulfilled.
\end{proof}

Finally, we have all necessary results to see that $\hat{g}_1^-$ is indeed our best bet. Define
\begin{align*}
	\hat{g}_1^-(z_\ell; w)
	&:= \min_{j \geq \ell} \max_{i \leq j} T^-(P_{i:j}^w)
	= \max_{i \leq \ell} \min_{j\geq i} T^-(P_{i:j}^w)\\
	\hat{g}_2^-(z_\ell;\hat{g}_1^-) 
	&:= -\min_{j \geq \ell} \max_{i \leq j} -\E(\bar{P}_{i:j}) 
	= -\max_{i \leq \ell} \min_{j\geq i} - \E(\bar{P}_{i:j}),
\end{align*}
where $\bar{P}_{i:j}$ is the empirical distribution of $L(\hat{g}_1^-(z_i),y_i),\dots,L(\hat{g}_1^-(z_j),y_j)$ and  $P^w_{i:j}$ is the weighted empirical distribution of $y_i,\dots,y_j$ with weights $w$.

\begin{proposition}\label{prop:minimizers}
Assume that there exist $\hat{g}_1, \hat{g}_2 \colon \{z_1, \dots, z_n\} \to \R$ minimizing \eqref{eq:ISOREG}, then $\hat{g}_1^-(\cdot; h(\hat{g}_2)), \hat{g}_2^-(\cdot; \hat{g}_1^-(\cdot; h(\hat{g}_2)))$ are also minimizers.
\end{proposition}
\begin{proof}
Clearly the pair $\hat{g}_1(\cdot), \hat{g}_2(\cdot)$ has to satisfy the restrictions imposed by Proposition \ref{prop:fixgi} as otherwise they would not be optimal. Proposition \ref{prop:main} implies that the pair $\hat{g}_1^-(\cdot; h(\hat{g}_2)), \hat{g}_2(\cdot)$ is also a minimizing pair to \eqref{eq:ISOREG}. Finally applying part (a) of Proposition \ref{prop:main} we can conclude that $\hat{g}_2(\cdot)= \hat{g}_2^-(\cdot; \hat{g}_1^-(\cdot; h(\hat{g}_2)))$.
\end{proof}

\subsection{Simultaneously optimal solutions}\label{sec:simultaneously}
A simultaneously optimal solution $\hat{g}_1, \hat{g}_2$ has to minimize the expected elementary losses 
\begin{align} \label{eq:elem1}
	\frac{1}{n}\sum_{i=1}^n S_{\eta,1}(g_1(z_i),y_i)
\end{align}
and
\begin{align}\label{eq:elem2}
	\frac{1}{n}\sum_{i=1}^n S_{\eta,2}(g_1(z_i),g_2(z_i),y_i) 
\end{align}
for all $\eta \in \R$ among all increasing functions $g_1: \{z_1, \dots, z_n\} \to \R$ and all decreasing functions $g_2: \{z_1, \dots, z_n\} \to \R$.
The expected elementary score \eqref{eq:elem1} is minimized for all $\eta \in \R$ if and only if $\hat{g}_1$ is an optimal isotonic solution with respect to $T$ characterized in \cite{Jordan2019}. Thus, there can only exist a simultaneously optimal solution if for one such $\hat{g}_1$ there exists $\hat{g}_2: \{z_1, \dots, z_n\} \to \R$ decreasing so that the pair $\hat{g}_1, \hat{g}_2$ minimizes \eqref{eq:elem2} for all $\eta \in \R$.

The proof of Proposition \ref{prop:main} suggests that for any $m \leq n$
\begin{align*}
 \sum_{i=m}^n L(\hat{g}_1^-(z_i), y_i)
 \leq 
 \sum_{i=m}^n L(\hat{g}_1(z_i), y_i) 
\end{align*} 
with equality whenever $m=n$. Note that minimizing \eqref{eq:elem2} for all $\eta \in \R$ is equivalent to minimizing
\begin{align*}
	\sum_{i=1}^n \one\{\eta \leq -g_2(z_i)\} L(g_1(z_i),y).
\end{align*}
Thus, a pair $\hat{g}_1, \hat{g}_2$ can only be simultaneously optimal if $\restr{\hat{g}_1}{m:n}$ is an optimal isotonic solution on $(z_m, y_m), \dots, (z_n,y_n)$ for all $m  \in \{1, \dots, n\}$ with $\hat{g}_2(z_{m-1})>\hat{g}_2(z_{m})$. If this is not the case for some $m \in \{1, \dots, n\}$, we can find $\hat{g}_1^m$ such that the pair $\hat{g}_1^m, \hat{g}_2^m$, where $\hat{g}_2^m$ is the corresponding solution obtained via Proposition \ref{prop:fixgi}, dominates $\hat{g}_1, \hat{g}_2$ for all $\eta \in \R$ with $\hat{g}_2(z_{m-1}) < \eta \leq \hat{g}_2(z_{m-1})$. But inevitably this solution performs worse for other $\eta \in \R$, especially for $\eta \leq -g_2(z_1)$. Figure \ref{fig:2biv} displays a data example where a simultaneously optimal solution does not exist because there exists some index $m$ with $\hat{g}_2(z_{m-1})>\hat{g}_2(z_{m})$ but $\restr{\hat{g}_1^-}{m:n}$ is not an optimal isotonic solution on $(z_m, y_m), \dots, (z_n,y_n)$.
The previous considerations are summarized by the following proposition. 
\begin{proposition}\label{prop:criterion}
A simultaneously optimal solution exists if and only if $\restr{\hat{g}_{1}^-}{m:n}$ is an optimal solution on $(z_m,y_m), \dots, (z_n,y_n)$ for all $m \in \{2, \dots,n\}$ such that $\hat{g}_2^-(z_{m-1}) > \hat{g}_2^-(z_m)$ and $ \hat{g}_1^-(z_{m-1}) = \hat{g}_1^-(z_m)$.
\end{proposition}
Proposition \ref{prop:criterion} supplies us with a criterion to check for simultaneous optimality. 
The approach is to first calculate 
\begin{align*}
	\hat{g}_1^-(z_\ell)&
	:= \min_{j \geq \ell} \max_{i \leq j} T^-(P_{i:j})
	= \max_{i \leq \ell} \min_{j\geq i} T^-(P_{i:j}),\\
	\hat{g}_2^-(z_\ell) 
	&:= -\min_{j \geq \ell} \max_{i \leq j} -\E(\bar{P}_{i:j})
	= -\max_{i \leq \ell} \min_{j\geq i} - \E(\bar{P}_{i:j}),
\end{align*}
with $\bar{P}$ as defined in Proposition \ref{prop:fixgi}.
In a second step, for each $m\geq 2$ with $\hat{g}_2^-(z_{m-1}) > \hat{g}_2^-(z_m)$ and $ \hat{g}_1^-(z_{m-1}) = \hat{g}_1^-(z_m)$ one has check whether $\restr{\hat{g}_1^-}{m:n}$ is an optimal solution on the subset $(z_m, y_m),\dots, (z_n,y_n)$. To check whether $\restr{\hat{g}_1^-}{m:n}$ remains optimal we can compare the expected elementary score for $\restr{\hat{g}_1^-}{m:n}$ to the one of $\hat{g}_{1;m:n}^-$.
If $\restr{\hat{g}_1^-}{m:n}$ remains optimal for each $m\geq 2$ with $ \hat{g}_2^-(z_{m-1}) > \hat{g}_2^-(z_m)$ and $ \hat{g}_1^-(z_{m-1}) = \hat{g}_1^-(z_m)$, then the solution $(\hat{g}_1^-, \hat{g}_2^-)$ is indeed simultaneously optimal.

For bivariate functionals $T$ with two elicitable components there always exists a subclass $\mathcal{L}_2$ of consistent loss functions $L(x_1,x_2,y)$ that are separable in the sense that $L(x_1,x_2,y) = L_1(x_1,y) + L_2(x_1,y)$.
Solving the isotonic regression problem simultaneously over all $L \in \mathcal{L}_2$ can be split into two independent optimization problems. In this case \cite{Jordan2019} provide all necessary tools for a complete characterization of all solutions.
But not all consistent loss functions lie necessarily in $\mathcal{L}_2$. If $T$ is a vector of moments this can be seen in Proposition 4.11 in \cite{Fissler2019}. In the case where $T$ is a vector of quantiles, however, $\mathcal{L}_2$ comprises all consistent losses \citep[Proposition 4.2]{Fissler2016} explaining some of the optimality properties of the IDR introduced by \cite{Henzi2019}. Thus, 
when considering functionals with elicitable components one can reach simultaneous optimality at least with respect to the class $\mathcal{L}_2$. When considering functionals with elicitation complexity greater than one however, there are no separable consistent loss functions, so that possibly no simultaneous optimum exists.

\begin{figure}
\centering
\includegraphics[width=\textwidth]{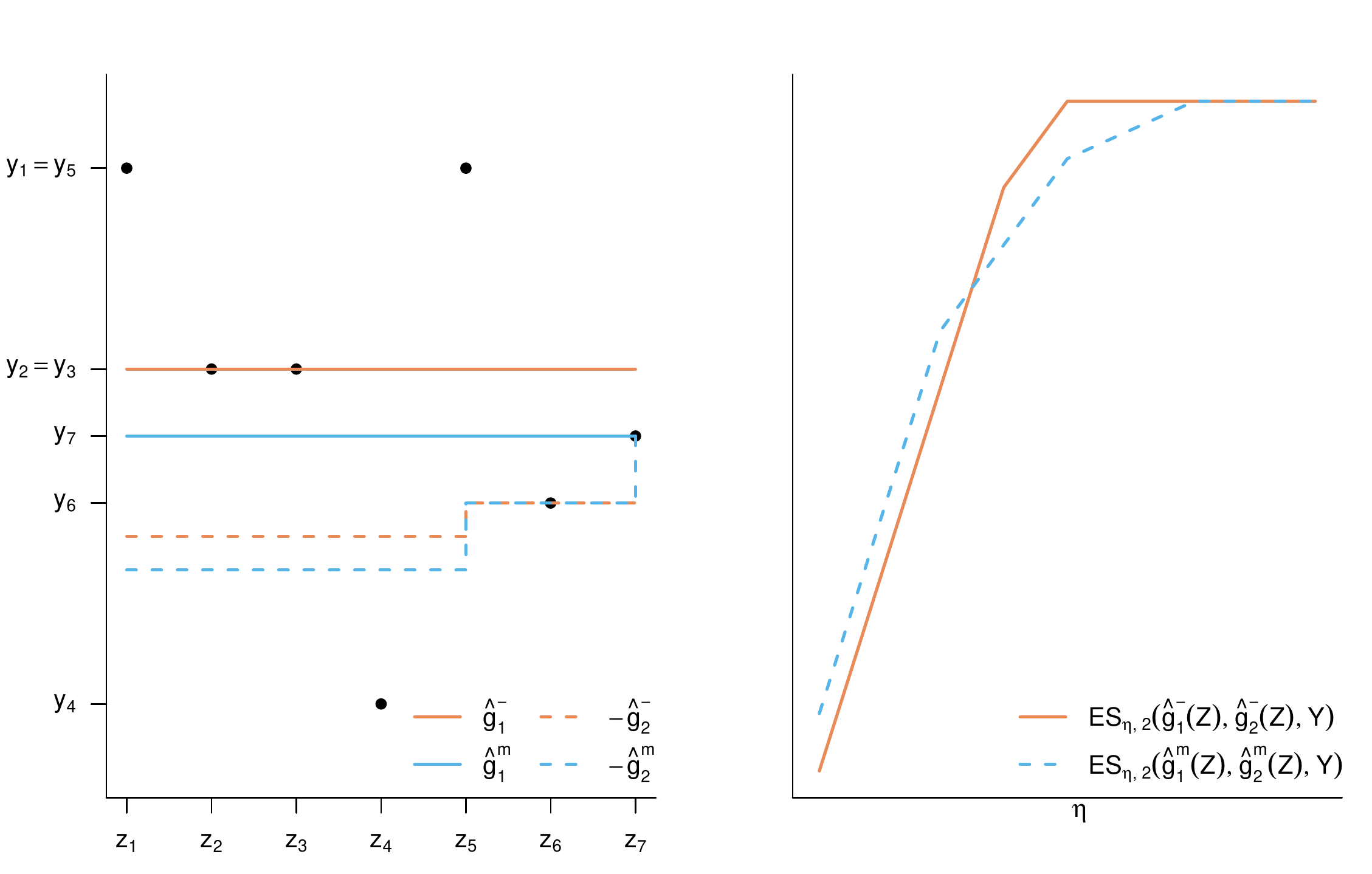}
\caption{Specific sample of seven data points (black) on the left, such that for $\underline{T}=(q_{0.5}, \ES_{0.5})$, $\hat{g}_{1}^- |_{5:7}$ is not an optimal isotonic solution on $(z_5, y_5)$, $(z_6, y_6)$, $(z_7, y_7)$ but $\hat{g}_2^-(z_4)>\hat{g}_2^-(z_5)$. The function $\hat{g}_1^m$ is not an optimal isotonic solution to the global optimization problem. The Murphy diagram \citep{Ehm2016} (plot of expected elementary scores) on the right shows the there are values of $\eta$ where $\hat{g}_1^m$, $\hat{g}_2^m$ has smaller expected loss than $\hat{g}_1^-$, $\hat{g}_2^-$.}
\label{fig:2biv}
\end{figure}

\section{Numerical experiments}\label{sec:simulation}
We let
\begin{align*}
	\begin{split}
	\hat{g}_1^-(z_\ell)&
	:= \min_{j \geq \ell} \max_{i \leq j} T^-(P_{i:j})
	= \max_{i \leq \ell} \min_{j\geq i} T^-(P_{i:j})\\
	\hat{g}_2^-(z_\ell) 
	&:= -\min_{j \geq \ell} \max_{i \leq j} -\E(\bar{P}_{i:j})
	= -\max_{i \leq \ell} \min_{j\geq i} - \E(\bar{P}_{i:j})
	\end{split}
\end{align*}
In this section we investigate on how often simultaneous optimality occurs and the number of iterations needed to obtain an optimal solution for a specific loss function, whenever the solution $\hat{g}_1^-, \hat{g}_2^-$ is not simultaneously optimal.
We consider the two prominent examples $(q_\alpha, \ES)$ and $(\E, \var)$ in the simulations.

First, let us examine what we would expect to result from those simulations in terms of simultaneous optimality.
In Section \ref{sec:solutions}, we saw that simultaneous optimality is attained whenever $\restr{\hat{g}_1^-}{m:n}$ remains an optimal solution for all $m \in \{2, \dots, n\}$ with $\hat{g}_2^-(z_{m-1})>\hat{g}_2^-(z_m)$. 
Clearly, this requirement is fulfilled as long as $\hat{g}_2^-$ jumps at the same point as $\hat{g}_1^-$. 
Naturally, the more jumps $\hat{g}_1^-$ has, or equivalently the less pooling was required, the higher are the chances for simultaneous optimality, in that there are not many additional restrictions left to be imposed by $\hat{g}_2^-$.
Thus, the less the isotonicity constraint is violated in the data the higher the chances for the pair $(\hat{g}_1^-, \hat{g}_2^-)$ to be simultaneously optimal. 
Only considering the impact of $\hat{g}_1^-$, we would expect the chance for simultaneous optimality to decrease with increasing variance in the data. Moreover, for fixed variance we would expect the chance of simultaneous optimality to decrease with increasing sample size, because the chance for necessary pooling increases. 

Concerning the impact of $\hat{g}_2^-$, we have seen in Proposition \ref{prop:fixgi} that $\hat{g}_2^-$ is fitted to the transformed data points $(z_1,L(\hat{g}_1^-(z_1),y_1)), \dots, (z_n,L(\hat{g}_1^-(z_n),y_n))$, where the transformed $y$-values depend on the loss $L$ of $y_\ell$ and $\hat{g}_1^-(z_\ell)$. 
The order sensitivity of the loss function ensures that the transformation $L(\hat{g}_1^-(z_\ell),y_\ell)$ takes larger values when $\hat{g}_1^-(z_\ell)$ and $y_\ell$ are far apart and smaller values when they are close.
Thus, if small modifications are necessary to obtain $\hat{g}_1$, then we would expect the transformed data to be approximately constant. The outcome of the transformation however depends on how the loss $L$ weighs the differences. 

The setup for the simulations was the following:  
For the pair $(q_\alpha, \ES_\alpha)$ we aimed to optimally fit an increasing function $\hat{g}_1^-$ and decreasing function $\hat{g}_2^-$ to simulated data sets. We drew $n$ points $z_\ell$ independently and uniformly from $[0, 100]$. The corresponding $y$-value was $y_\ell=z_\ell + \epsilon_\ell$ where $\epsilon_\ell \sim \mathcal{N}(0, \sigma^2)$ are independent and independent of $z_\ell$. We let $n \in \{10,100,500,1000\}$ and $\sigma \in \{3,10,20,30\}$ and we repeated the experiment $M=1000$ times to count the number of times simultaneous optimality occurred. 
To investigate whether the results differ depending on level $\alpha$, we calculated $\hat{g}_1^-$ and $\hat{g}_2^-$ for each data set for all $\alpha \in \{0.1,0.2,0.3,0.4,0.5,0.6,0.7,0.8,0.9\}$. For a specific data set, Figure \ref{fig:7pav} shows the fits $\hat{g}_1^-$ for all levels $\alpha$, and Figure \ref{fig:8pav} contains the corresponding fits $-\hat{g}_2^-$. 

For the pair $(\E, \var)$ we aimed to optimally fit two increasing functions $\hat{g}_1^-$ and $\hat{g}_2^-$ to simulated data sets. All our results can clearly be adapted to this case. Thus, again drew $n$ points $z$ independently and uniformly from $[0, 100]$. The corresponding $y$-value was $y_\ell= z_\ell + \epsilon_\ell$ where $\epsilon_\ell \sim \mathcal{N}(0, c {\ell}/{\sqrt{n}})$ were independent. We let $n \in \{10,100,500,1000\}$ and $c \in \{0.5,1,3,6\}$ and then generated $M=1000$ data sets and calculated the corresponding fits $\hat{g}_1^-$ and $\hat{g}_2^-$. 
Figure \ref{fig:9biv} contains the fits $\hat{g}_1^-$ and $\hat{g}_2^-$ for a specific data set.

Using the criterion in Proposition \ref{prop:criterion}, we counted the number times simultaneous optimality occurred. Table \ref{table:2} contains the results for the pair $(q_\alpha, \ES_\alpha)$. The percentage of times simultaneous optimality is reached is displayed. The results confirm our expectations.
With increasing sample size and increasing variance the percentage decreases drastically. 
The reason that not all levels $\alpha$ are equally affected is due to the different weights that $L$ imposes depending on the level $\alpha$.

The results for the pair $(\E, \var)$ in Table \ref{table:1} also confirm our expectations. The reason why the percentage in this case decreases even more rapidly is that the expectation $\E$ is less robust when it comes to removing data from a partition element than the quantile $q_\alpha$ is.


\begin{figure}[ht] 
\centering
\includegraphics[width=.9\textwidth]{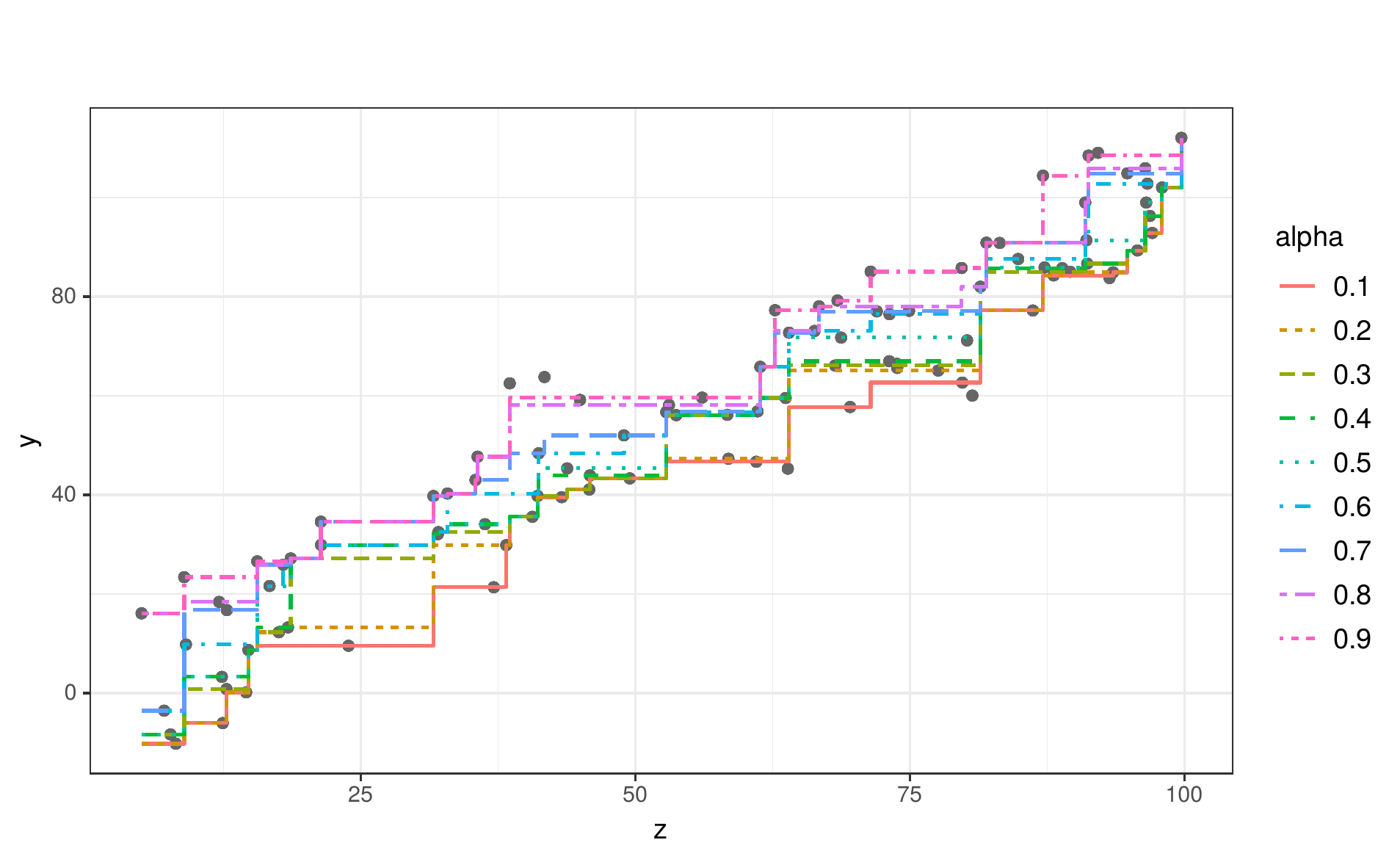}
\caption{For a set of $n=100$ data points and the pair $(q_\alpha, \ES_\alpha)$ the optimal fit $\hat{g}_1^-$ was drawn for each $\alpha \in \{0.1,0.2,0.3,0.4,0.5,0.6,0.7,0.8,0.9\}$.}
\label{fig:7pav}
 \vspace*{\floatsep}
\includegraphics[width=.9\textwidth]{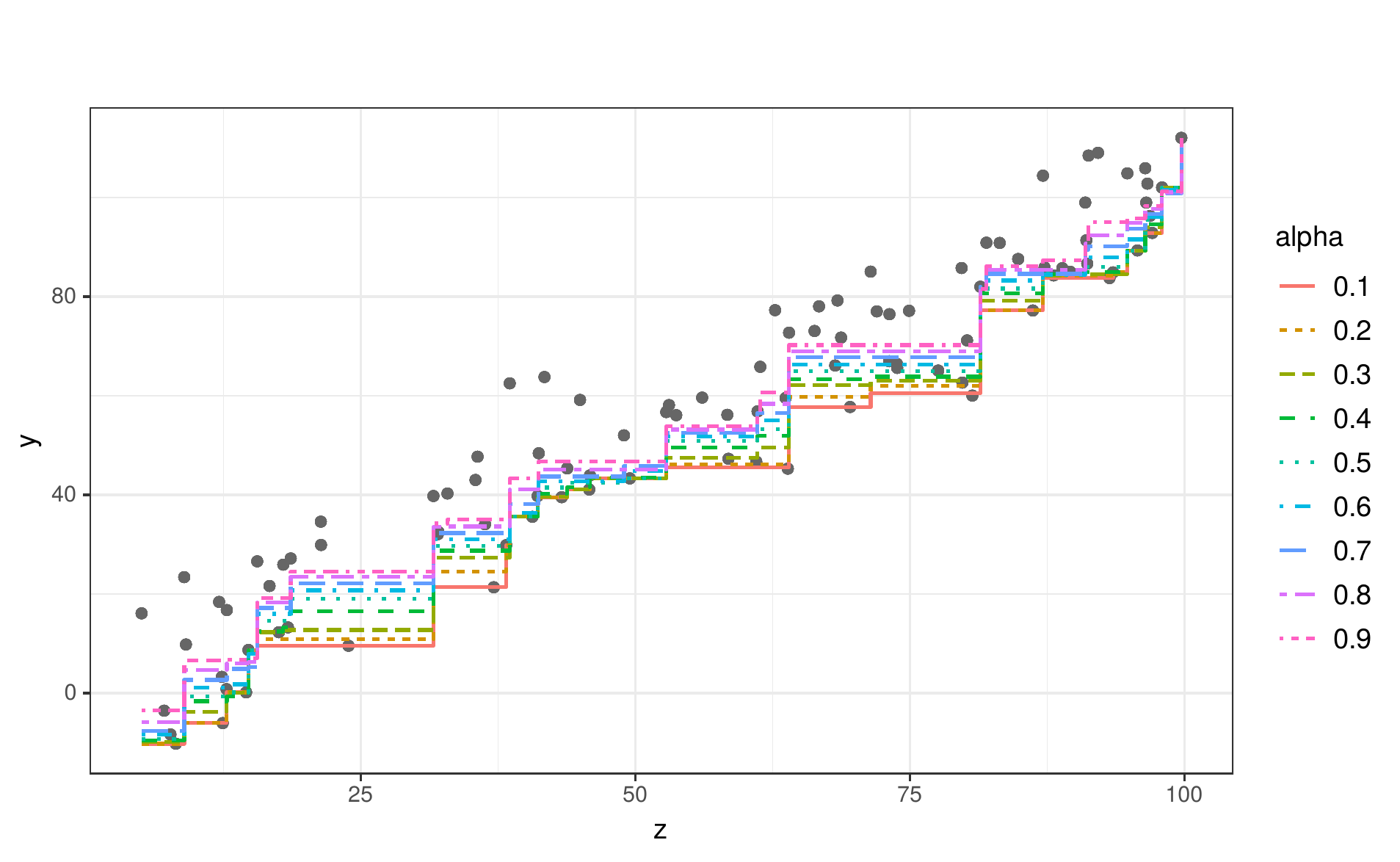}
\caption{For the same choice of $n=100$ data points as in Figure \ref{fig:7pav} the corresponding fits $\hat{g}_2^-$ are calculated and $-\hat{g}_2^-$ is displayed for each $\alpha \in \{0.1,0.2,0.3,0.4,0.5,0.6,0.7,0.8,0.9\}$.}
\label{fig:8pav}
\end{figure}


\begin{figure}[ht] 
\centering
\includegraphics[width=.9\textwidth]{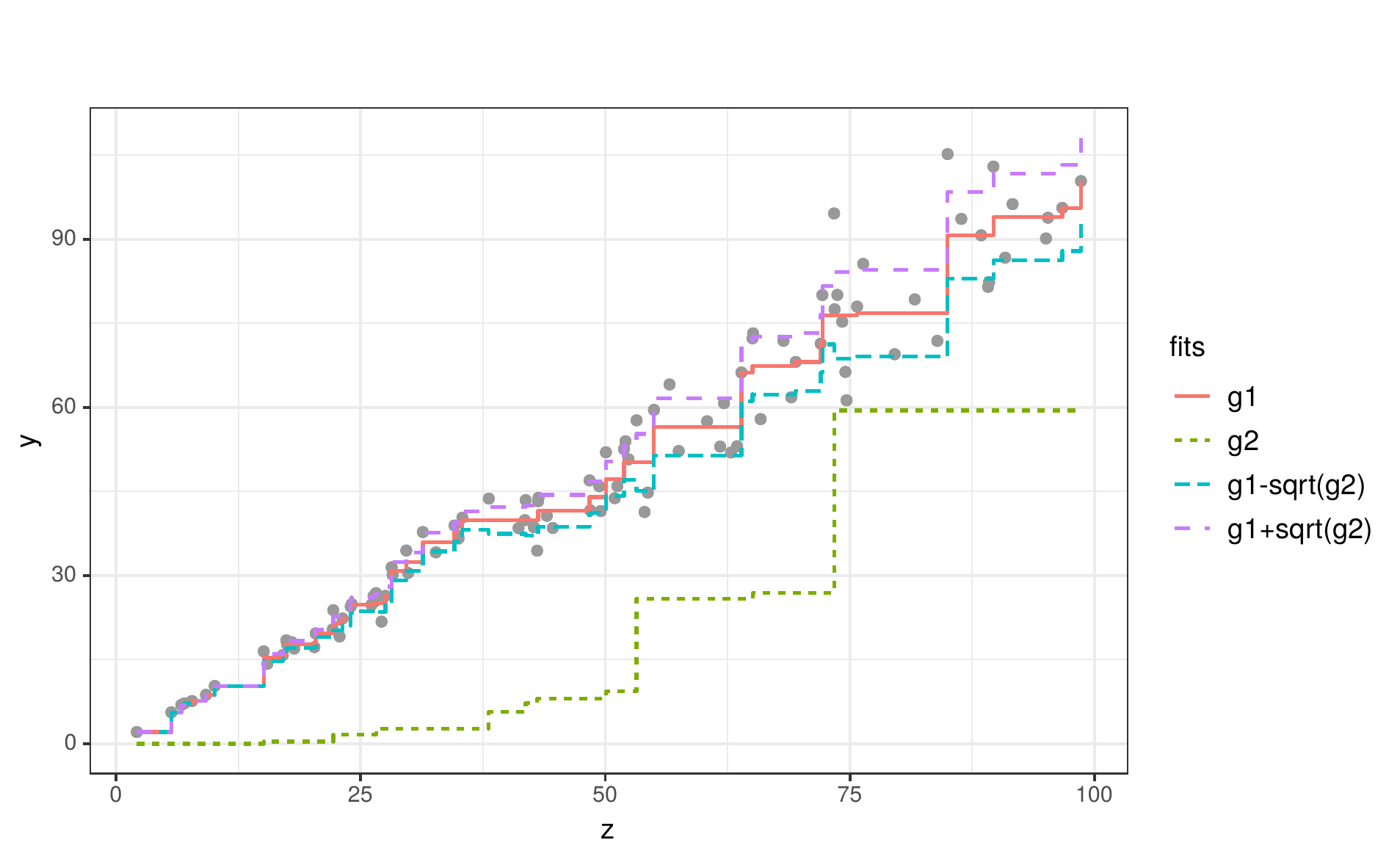}
\caption{For a sample of $n=100$ data points and the pair $(\E, \var)$ the optimal fit $\hat{g}_1$ is drawn in red and ${\hat{g}_2}$ is in green. Moreover, $\hat{g}_1-\sqrt{\hat{g}_2}$ and $\hat{g}_1+\sqrt{\hat{g}_2}$ are drawn in blue and pink, respectively}
\label{fig:9biv}
\end{figure}

\begin{landscape}
\begin{table}[p] 
\caption{Percentage of times simultaneous optimality occurred for $(q_\alpha, \ES_\alpha)$ for each combination of sample size $n$, standard deviation $\sigma$, and level $\alpha$.}  
\label{table:2}
\vspace{0.1cm}
\centering 
\begin{tabular}{ll|rrrrrrrrr} 
\hline 
 & & $\alpha=0.1$ & $\alpha=0.2$ & $\alpha=0.3$ & $\alpha=0.4$ & $\alpha=0.5$ & $\alpha=0.6$ & $\alpha=0.7$ & $\alpha=0.8$ & $\alpha=0.9$ \\ 
\hline
$n=10$& $\sigma=3$ & 1.00 & 1.00 & 1.00 & 1.00 & 1.00 & 1.00 & $0.93$ & $0.94$ & $0.94$ \\ 
& $\sigma=10$ &  1.00 & 1.00 & 1.00 & 1.00 & $0.98$ & $0.96$ & $0.79$ & $0.70$ & $0.69$ \\ 
& $\sigma=20$ & 1.00 & 1.00 & 1.00 & $0.98$ & $0.94$ & $0.92$ & $0.69$ & $0.58$ & $0.50$ \\ 
& $\sigma=30$ &  1.00 & 1.00 & $0.98$ & $0.97$ & $0.88$ & $0.88$ & $0.64$ & $0.53$ & $0.44$ \\  \hline
$n=100$ & $\sigma=3$ & 1.00 & $0.97$ & $0.60$ & $0.48$ & $0.14$ & $0.13$ & 0.00 & 0.00 & 0.00 \\ 
& $\sigma=10$ &$0.96$ & $0.53$ & $0.16$ & $0.08$ & $0.01$ & $0.01$ & 0.00 & 0.00 & 0.00 \\ 
& $\sigma=20$ & $0.80$ & $0.31$ & $0.11$ & $0.06$ & $0.01$ & $0.02$ & 0.00 & 0.00 & 0.00 \\ 
& $\sigma=30$ & $0.70$ & $0.27$ & $0.12$ & $0.06$ & $0.02$ & $0.02$ & 0.00 & 0.00 & 0.00 \\  \hline
$n=500$ & $\sigma=3$ & $0.47$ & 0.00 & 0.00 & 0.00 & 0.00 & 0.00 & 0.00 & 0.00 & 0.00 \\ 
& $\sigma=10$ & $0.06$ & 0.00 & 0.00 & 0.00 & 0.00 & 0.00 & 0.00 & 0.00 & 0.00 \\ 
& $\sigma=20$ & $0.03$ & 0.00 & 0.00 & 0.00 & 0.00 & 0.00 & 0.00 & 0.00 & 0.00 \\  
& $\sigma=30$ &$0.04$ & 0.00 & 0.00 & 0.00 & 0.00 & 0.00 & 0.00 & 0.00 & 0.00 \\  \hline
$n=1000$ & $\sigma=3$ & $0.01$ & 0.00 & 0.00 & 0.00 & 0.00 & 0.00 & 0.00 & 0.00 & 0.00 \\ 
& $\sigma=10$ & 0.00 & 0.00 & 0.00 & 0.00 & 0.00 & 0.00 & 0.00 & 0.00 & 0.00 \\ 
& $\sigma=20$ & 0.00 & 0.00 & 0.00 & 0.00 & 0.00 & 0.00 & 0.00 & 0.00 & 0.00 \\ 
& $\sigma=30$ &  0.00 & 0.00 & 0.00 & 0.00 & 0.00 & 0.00 & 0.00 & 0.00 & 0.00 \\ 
\hline \\[-1.8ex] 
\end{tabular}
\end{table} 
\end{landscape}

\begin{table}[ht!]
  \caption{Percentage of times simultaneous optimality occurred for $(\E, \var)$ for each combination of sample size $n$ and constant $c$.} 
\label{table:1}
\vspace{0.1cm}
 \centering
\begin{tabular}{l|rrrr} 
\hline 
 & $c=0.5$ & $c=1$ & $c=3$ & $c=6$ \\ 
\hline 
$n=10$ & $0.98$ & $0.95$ & $0.79$ & $0.57$ \\ 
$n=100$ & $0.02$ & $0.00$ & $0.00$ & $0.00$ \\ 
$n=500$ & $0.00$ & $0.00$ & $0.00$ & $0.00$ \\ 
$n=1000$ & $0.00$ & $0.00$ & $0.00$ & $0.00$ \\ 
\hline
\end{tabular}  
\end{table}

Simultaneous optimality is usually not attainable. In these cases, we have to choose a specific loss function to solve the isotonic regression problem. It is natural to ask, how different these solutions are compared to our candidate for simultaneous optimality. 

For both examples, we choose two different functions $h$ and count the number of iterations the algorithm needed to get from the candidate for simultaneous optimality to a potential optimal solution for the specific loss. For the pair $(q_\alpha, \ES_\alpha)$, we considered the (1/2)-homogeneous loss from \cite{Nolde2017}. It arises when choosing $h$ in \eqref{eq:loss} as $h_1(x)=1/(2 \sqrt{x})$. We also considered $h_2(x)=\exp(-x)$. The iteration was stopped when the loss given by \eqref{eq:ISOREG} did not improve by more that $10^{-10}$. For both loss functions, almost no adjustments were necessary with a maximum average number of iterations for $h_1$ of $0.11$ when $\sigma = 30$ and $\alpha=0.3$, and for $h_2$ of $0.05$ when $\sigma=30$ and $\alpha=0.1$. For most combinations of $\sigma$ and $\alpha$, the average number of iterations was zero for both loss functions which is why detailed results are not displayed. This suggests that although, the candidate for simultaneous optimality is not simultaneously optimal, it still is optimal with respect to several losses.

For the pair $(\E, \var)$, we chose functions $h_1(x)=1/(x+0.1)$ and $h_2(x)=\exp(-x/50+0.1)$. The reason for dividing by $50$ was the scale of the weights to avoid numerical issues. The summand $+0.1$ was to avoid weights of zero. Again, the iteration was stopped when the loss given by \eqref{eq:ISOREG} did not improve by more that $10^{-10}$. Figure \ref{fig:iterated} displays the corresponding solutions obtained for a specific data set. The average number of iterations is displayed in Table \ref{table:4}. Here, the situation is different. Given a specific loss function for the pair $(\E,\var)$, the global loss may decrease through adaptations of the optimal solution for $\E$ alone.

\begin{table}[ht!]
\caption{The average number of iterations are displayed for the two weight functions $h_1, h_2$ considered for the pair $(\E,\var)$.} 
\label{table:4} 
\vspace{0.1cm}
\centering 
\begin{tabular}{ll|rrrr} 
\hline 
 && $c=0.5$ & $c=1$ & $c=3$ & $c=6$ \\ 
\hline 
$n=10$ & $h_1$ &$0.07$ & $0.24$ & $1.15$ & $2.77$ \\ 
 & $h_2$ &$0.01$ & $0.06$ & $0.94$ & $2.79$ \\  \rowcolor{gray!25}
$n=100$ & $h_1$ & $10.52$ & $12.48$ & $13.95$ & $14.45$ \\ \rowcolor{gray!25}
& $h_2$ & $2.90$ & $6.52$ & $13.12$ & $12.73$ \\  
$n=500$ & $h_1$ & $10.54$ & $11.76$ & $13.77$ & $13.58$ \\  
& $h_2$ & $6.68$ & $11.05$ & $14.40$ & $8.78$ \\ \rowcolor{gray!25}
$n=1000$ & $h_1$ & $9.16$ & $10.05$ & $11.37$ & $12.58$  \\ \rowcolor{gray!25}
& $h_2$ & $7.91$ & $12.04$ & $12.19$ & $3.32$ \\ 
\hline
\end{tabular} 
\end{table} 

\begin{figure}[H]
\includegraphics[width=\textwidth]{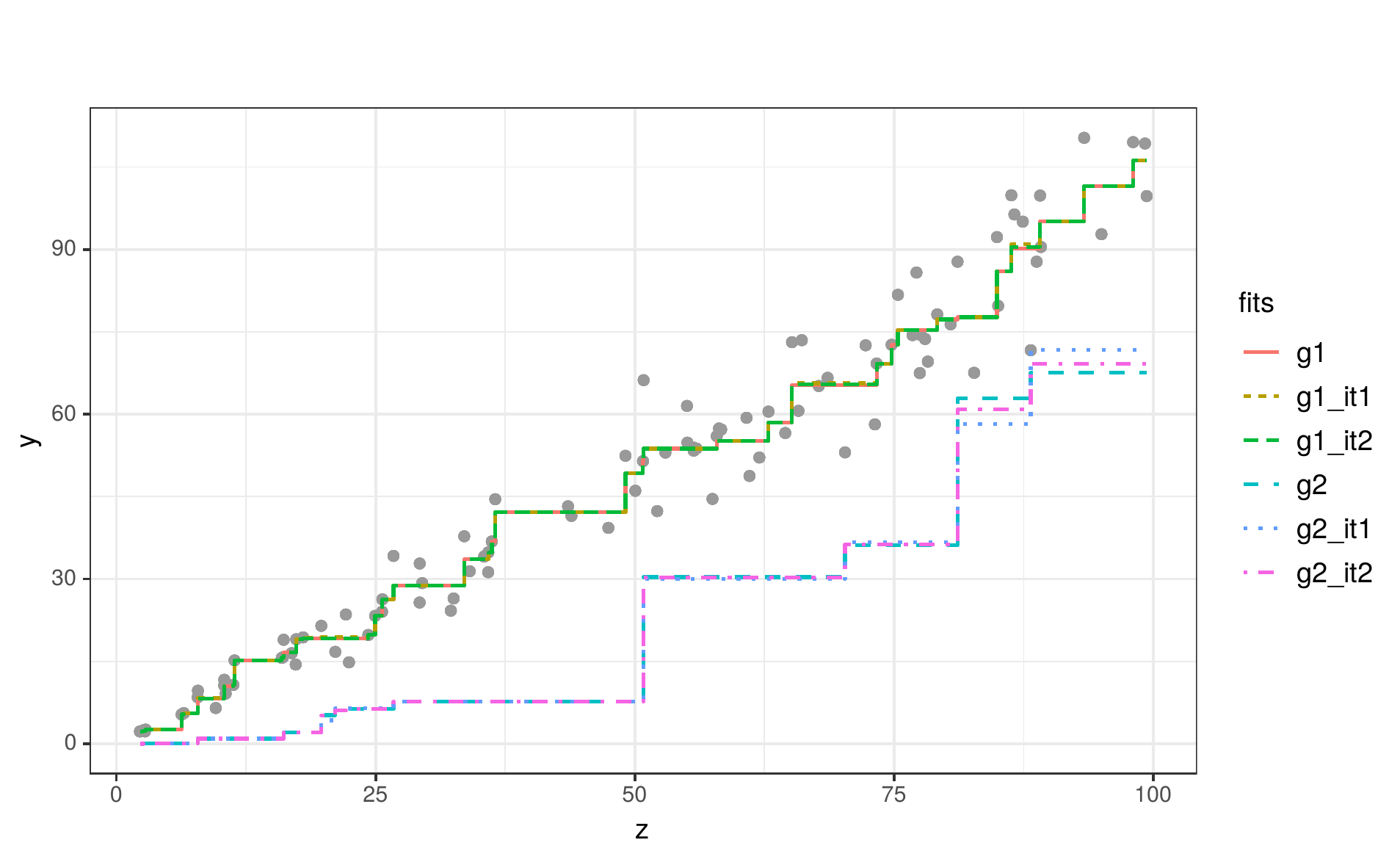}
\caption{For a specific sample of size $100$ the original fits (g1 and g2) are displayed in red and light blue respectively. The light green and dark blue fits (g1$\_$it1 and g2$\_$it1) correspond to the iterated versions of $g_1$ and $g_2$, respectively, with respect to weight function $h_1$. Finally, the dark green  and the pink fits (g1$\_$it2 and g2$\_$it2) correspond to  the iterated versions of $g_1$ and $g_2$, respectively, with respect to $h_2$. For $h_1$ the number of iterations was $13$ and for $h_2$ a total of $5$ iterations were necessary.}
\label{fig:iterated}
\end{figure}

\bibliographystyle{abbrvnat}
\bibliography{biblio}

\appendix
\section{Generalizations to partial orders}
The results in this article can be generalized to partially ordered covariate sets.
Let distribution $P$ be the distribution of the random vector $(Z, Y) \in \mathcal{Z} \times \R$, where $\mathcal{Z}$ is a finite partially ordered set. We denote the partial order by $\preceq$.
We aim now to minimize the criterion
\begin{multline}\label{eq:ISOREG2}
\int \tilde{L}(g_1(z),g_2(z),y) \, P(\diff z, \diff y)
\\= \int \Big(H(g_2(z)) + h(g_2(z))\big(L(g_1(z),y) - g_2(z)\big)\Big) \, P(\diff z, \diff y)
\end{multline}
among all increasing functions $g_1\colon \mathcal{Z} \to \R$ and decreasing $g_2\colon \mathcal{Z} \to \R$, that is, for $z \preceq z'$ we have $g_1(z) \leq g_1(z')$ and $g_2(z) \geq g_2(z')$.
We call any minimizing pair a solution to the isotonic regression problem.
Following \cite{Jordan2019}, in order to accommodate the partially ordered set $\mathcal{Z}$, we introduce upper sets $x \subseteq \mathcal{Z}$ to replace single indices $i \in \{1, \dots, n + 1\}$. Set $x$ is said to be an \emph{upper set} if $z\in x$ and $z \preceq z'$ implies $z' \in x$. We denote $P_x(A)=P((x \times \R) \cap A)$ for any $A \in \mathcal{P}(\mathcal{Z}) \otimes \mathcal{B}(\R)$, where $\mathcal{B}(\R)$ denotes the Borel $\sigma$-algebra on $\R$. Let $\mathcal{X}$ consist of all admissible superlevel sets for an increasing function $g$ imposed by the partial order on $\mathcal{Z}$.

As in the case of total orders, keeping either $g_1$ or $g_2$ fixed, we can find the optimal solution to \eqref{eq:ISOREG2} with respect to the other component.

\begin{proposition}\label{prop:fixed2}
\begin{enumerate}
\item[(a)] Let $g_1:\mathcal{Z}\to \R$ be given. Then, the optimal antitonic solution $\hat{g}_2$ of \eqref{eq:ISOREG} corresponding to $\hat{g}_1$ is given by
\begin{align*}
		\hat{g}_2(z) 
	&= -\min_{x':z\notin x'} \max_{x \supsetneq x'} -\E(\bar{P}_{x \setminus x'})
	= -\max_{x: z \in x} \min_{x' \subsetneq x} - \E(\bar{P}_{x \setminus x'}),
\end{align*}
where $\bar{P}_{i:j}$ is the empirical distribution of $L(g_1(z_i),y_i),\dots,L(g_1(z_j),y_j)$.
\item[(b)] Let $g_2: \mathcal{Z} \to \mathbb{R}$ be given. Then, any optimal isotonic solution $\hat{g}_1$ of \eqref{eq:ISOREG} with $g_2$ fixed satisfies
\[
\min_{x':z\notin x'} \max_{x \supsetneq x'}  T^-({P}_{x \setminus x'}^w) 
\leq \hat{g}_1(z) 
\leq \max_{x: z \in x} \min_{x' \subsetneq x} T^+(P^w_{x \setminus x'}),
\]
where $P^w_{x \setminus x'}$ is the weighted empirical distribution of $y$ with $z \in x \setminus x'$ and weights proportional to $h(g_2(z)),z \in \mathcal{Z}$.
\end{enumerate}
\end{proposition}
\begin{proof}
Follows with the same argument as for total orders.
\end{proof}

As in Section \ref{sec:solutions}, we need to introduce some notation for the investigations ahead. In the following, we denote an optimal solution on the subset $\bar{x} \subseteq \mathcal{Z}$ by $\hat{g}_{1;\bar{x}}$ and by $\restr{\hat{g}_1}{\bar{x}}$ we denote the optimal solution on the original set redistricted to $\bar{x}$.

Thinking in terms of superlevel sets, Lemma \ref{lem:order} states that $\restr{\hat{g}_1}{\mathcal{Z}\setminus\bar{x}} \leq \hat{g}_{1;\mathcal{Z}\setminus\bar{x}}$ for any $\bar{x} \in \mathcal{X}$.
Equivalently, $\restr{\hat{g}_1}{\bar{x}} \geq \hat{g}_{1;\bar{x}}$.
\begin{lemma}\label{lem:order2} Let $\bar{x} \in \mathcal{X}$ and assume that
\begin{align*}
	\hat{g}_1(z)&
	:= \min_{x':z\notin x'} \max_{x \supsetneq x'}  
	T^\lambda({P}_{x \setminus x'}^w) 
	= \max_{x: z \in x} \min_{x' \subsetneq x} 
	T^\lambda(P^w_{x \setminus x'})
\end{align*}
for some $\lambda \in [0,1]$. Then we have $\restr{\hat{g}_1}{\bar{x}} \geq \hat{g}_{1;\bar{x}}$.
\end{lemma}
\begin{proof}
It suffices to notice that
\begin{equation*}
	\restr{\hat{g}_1}{\bar{x}} (z)
	= \max_{x: z \in x} \min_{x' \subsetneq x} 
	T^\lambda(P^w_{x \setminus x'})
	\geq
	\mathop{\max_{x \in \mathcal{X};\, x  \subseteq  \bar{x} ;}}_{ z \in x} 
	\min_{x' \in \mathcal{X};\, x' \subsetneq x} 
	T^\lambda(P^w_{x \setminus x'})
	= \hat{g}_{1;\bar{x}}(z).\qedhere
\end{equation*}
\end{proof}

Let us recall the following observations made in \cite{Jordan2019}. For fixed weights $w$, we can minimize 
\begin{align}\label{eq:index2}
	\int_{x \times \R} V(\eta,y) \, P^w(\diff y),
	\quad \text{for all }
	\eta \in \R
\end{align}
among all admissible superlevel sets $x$ for an increasing function $g_1: \mathcal{Z} \to \R$ to obtain an optimal solution to \eqref{eq:ISOREG2}. The search for the optimal superlevel set $x$ needs to be conducted for every $\eta \in \R$. 
 Again there is a one-to-one correspondence between admissible superlevel sets and optimal solutions. Instead of an increasing function $\iota: \R \to \{1, \dots, n+1\}$ with $\iota(\eta) \in I_{1:n}(\eta)$ for all $\eta$, we now have a decreasing function $\xi: \R \to \mathcal{Z}$, in the sense that $\xi(\eta')\subseteq \xi(\eta)$ for $\eta' > \eta$. Moreover, it should hold that $\xi(\eta) \in X_{\mathcal{Z}}(\eta)$ for all $\eta \in \R$, where $X_{\mathcal{Z}}(\eta)\subseteq \mathcal{X}$ denotes the set of all superlevel sets minimizing \eqref{eq:index2}. Then the correspondence between an optimal solution $\hat{g}_1$ and $\xi(\eta)$ is given by
 \begin{align*}
 	\inf\{\eta:z \notin \xi(\eta)\}=\hat{g}_1(z)=\max \{ \eta : z \in \xi(\eta)\}.
 \end{align*}

The next result is the generalization of Lemma \ref{lem:index} to partial orders.
\begin{lemma}\label{lem:index2}
Let $\bar{x} \in \mathcal{X}$. We have that $X_{\mathcal{Z}}(\eta) \cap (\mathcal{Z}\setminus\bar{x}) \subseteq X_{\mathcal{Z}\setminus\bar{x}}(\eta)$, where $ X_{\bar{x}}(\eta)$ is the set of minimizing superlevel sets for the isotonic regression problem \eqref{eq:ISOREG2} on the subsample $(z,y), z \in \bar{x} \subseteq \mathcal{Z}$.
\end{lemma}
\begin{proof}
Let $x' \in X_{\mathcal{Z}}(\eta) \cap (\mathcal{Z}\setminus\bar{x})$ for some $\eta \in \R$. Therefore, the function 
\begin{align*}
	t_\eta : \mathcal{Z}\to\R,
	x \mapsto\int_{x \times \R} V(\eta,y) \, P^w(\diff y)
\end{align*}
has a minimum at $x'$. We can write
\begin{align*}
	\int_{x \times \R} V(\eta,y) \, P^w(\diff y)
	= \int_{x \cap (\mathcal{Z}\setminus\bar{x}) \times \R} V(\eta,y) \, P^w(\diff y)
	+ \int_{x \cap \bar{x} \times \R} V(\eta,y) \, P^w(\diff y).
\end{align*}
Hence, $\restr{t_\eta}{\mathcal{Z}\setminus\bar{x}}$ has a minimum at $x'$ and thus $x' \subseteq X_{\mathcal{Z}\setminus\bar{x}}(\eta)$. If $t_\eta$ has a minimum in $x=\bar{x}$, then
\begin{align*}
	t_\eta(x) - \int_{x \cap \bar{x} \times \R} V(\eta,y) \, P^w(\diff y) \geq 0,
\end{align*}
with equality in $x=\bar{x}$. Thus, $\emptyset \in X_{\mathcal{Z}\setminus\bar{x}}(\eta)$.
\end{proof}

Let us generalize Proposition \ref{prop:main} to partial orders.
\begin{proposition}\label{prop:main2}
For fixed ${g}_2$, corresponding $\hat{g}_1^-$ and any increasing $\hat{g}_1$ we have
	\begin{align*}
	\int \tilde{L}(g_1^-(z),g_2(z),y) \, P(\diff z, \diff y)
	\leq
	\int \tilde{L}(g_1(z),g_2(z),y) \, P(\diff z, \diff y).
	\end{align*} 
\end{proposition}
\begin{proof} Let $\mathcal{Q}$ and $\mathcal{Q}^-$ denote the partition of $\mathcal{Z}$ corresponding to $\hat{g}_1$ and $\hat{g}_1^-$, respectively.
By Lemma \ref{lem:mixture}, it suffices to show that for all $\eta \in \R$
\begin{align*}
		\int S_{\eta,2}(\hat{g}_1^-(z),g_2(z),y) \, P(\diff z, \diff y)
	\leq 
		\int S_{\eta,2}(\hat{g}_1(z),g_2(z),y)\, P(\diff z, \diff y).
\end{align*}
For the latter, it suffices to show that for all $\bar{x} \in \mathcal{X}$
\begin{align*}
	\int_{\bar{x} \times \R}
	L(\hat{g}_1^-(z), y) \, P(\diff z, \diff y)
	\leq
	\int_{\bar{x} \times \R}
	L(\hat{g}_1(z), y) \, P(\diff z, \diff y).
\end{align*}
Again it suffices to consider $\hat{g}_1$ with superlevel stets in $\cup_\eta X(\eta)$ and again  we will prove the converse. In other words, for all $\bar{x} \in \mathcal{X}$ we have
\begin{align}\label{eq:wts2}
	\int_{\mathcal{Z}\setminus \bar{x} \times \R}
	L(\hat{g}_1(z), y) \, P(\diff z, \diff y)
	\leq
	\int_{\mathcal{Z}\setminus \bar{x} \times \R}
	L(\hat{g}_1^-(z), y) \, P(\diff z, \diff y)
\end{align}

If $\mathcal{Z}\setminus \bar{x} = Q_1 \cup \dots \cup Q_{i}$, $Q_1, \dots, Q_{i} \in \mathcal{Q}$, Lemma \ref{lem:index2} implies that $\restr{\hat{g}_1}{\mathcal{Z}\setminus \bar{x}}$ is optimal on $(z,y), z \in \mathcal{Z}\setminus \bar{x}$. Thus, \eqref{eq:wts2} holds trivially.
If there exists no sequence of partition elements such that $\mathcal{Z}\setminus \bar{x} = Q_1 \cup \dots \cup Q_{i}$ we distinguish two cases.

\noindent
\textbf{Case 1:} If $\mathcal{Z}\setminus \bar{x} = Q_1^- \cup \dots \cup Q_{i^-}^-$, $Q_1^- , \dots , Q_{i^-}^- \in \mathcal{Q}^-$ Lemma \ref{lem:order2} implies that
\begin{align*}
	\restr{\hat{g}_1^-}{\mathcal{Z}\setminus \bar{x}}
	=\hat{g}_{1;\mathcal{Z}\setminus \bar{x}}^-
	\leq
	\restr{\hat{g}_1}{\mathcal{Z}\setminus \bar{x}}
	\leq
	\restr{\hat{g}_1^+}{\mathcal{Z}\setminus \bar{x}}
	\leq \hat{g}_{1;\mathcal{Z}\setminus \bar{x}}^+
\end{align*}
Moreover, by Lemma \ref{lem:index2}, $X_{\mathcal{Z}}(\eta) \cap (\mathcal{Z}\setminus\bar{x}) \subseteq X_{\mathcal{Z}\setminus\bar{x}}(\eta)$. Hence $\restr{\xi}{\mathcal{Z}\setminus\bar{x}}(\eta) \in X_{\mathcal{Z}\setminus\bar{x}}(\eta)$ for all $\eta \in \R$, where $\xi:\R \to \mathcal{Z}$ is the function imposing the score-minimizing superlevel sets corresponding to $\hat{g}_1$. Thus, by Proposition 4.5 in \cite{Jordan2019} $\restr{\hat{g}_1}{\mathcal{Z}\setminus \bar{x}}$ is an optimal solution to the isotonic regression problem on $(z,y), z \in \mathcal{Z}\setminus \bar{x}$.

\noindent
\textbf{Case 2:} It remains to consider the case where no sequence of partition elements such that $\mathcal{Z}\setminus \bar{x} = Q_1^- \cup \dots \cup Q_{i_-}^-$ exists.
Note that $\hat{g}_1$ is optimal for all $z\in \mathcal{Z}\setminus\bar{x}$ with $\hat{g}_{1;\mathcal{Z}\setminus\bar{x}}^-(z)\leq\hat{g}_1(z)$. Indeed, for those $z$, we have $\bar{g}_{1;\mathcal{Z}\setminus\bar{x}}^-(z) \leq \hat{g}_1(z)\leq \hat{g}_{1;\mathcal{Z}\setminus\bar{x}}^+(z)$, and can argue as in case 1.
For $z\in \mathcal{Z}\setminus\bar{x}$ with $\hat{g}_{1;\mathcal{Z}\setminus\bar{x}}^-(z)>\hat{g}_1(z)$, we can argue similarly as in the proof of Proposition \ref{prop:main}. 
For every $z \in \{z' \in \mathcal{Z}\setminus\bar{x} : \hat{g}_{1;\mathcal{Z}\setminus\bar{x}}^-(z')>\hat{g}_1(z')\}$ we have $z \in Q_{i+r}$, $r \in \{1,\dots,k\}$. Moreover, $\hat{g}_1$ is constant on every each $Q_{i+r}$, $r \in \{1,\dots,k\}$.
With the same reasoning as in the proof of Proposition \ref{prop:main}, we obtain that 
\begin{align*}
	\int_{Q_{i+r}^> \times \R}
	L(\hat{g}_{1;\mathcal{Z}\setminus\bar{x}}^-(z), y) \, 
	P(\diff z, \diff y)
	&\leq
	\int_{Q_{i+r}^> \times \R}
	L(c_i, y) \, 
	P(\diff z, \diff y) \\&
	\leq
	\int_{Q_{i+r}^> \times \R}
	L(c_i^-, y) \, 
	P(\diff z, \diff y)
\end{align*}
for all $r \in \{1,\dots,k\}$, where $Q_{i+r}^>:= Q_{i+r} \cap \{z \in \mathcal{Z}\setminus\bar{x} : \hat{g}_{1;\mathcal{Z}\setminus\bar{x}}^-(z)>\hat{g}_1(z)\}$. This implies the statement.
\end{proof}

Proposition \ref{prop:minimizers} also translates directly to partial orders.
\begin{proposition}\label{prop:minimizers2}
Assume that there exist $\hat{g}_1, \hat{g}_2 \colon \mathcal{Z} \to \R$ minimizing \eqref{eq:ISOREG2}, then $\hat{g}_1^-(\cdot; \hat{g}_2)$, and the corresponding $\hat{g}_2^-(\cdot; \hat{g}_1^-(\cdot; \hat{g}_2))$ are also minimizers.
\end{proposition}
\begin{proof}
The argument is the same as in the proof of Proposition \ref{prop:minimizers}.
\end{proof}

As in the case of total orders a simultaneously optimal solution may not necessarily exists, since $\hat{g}_2^-$ imposes additional constraints. Nonetheless, we are able to formulate a criterion so that simultaneous optimality is reached whenever the criterion is fulfilled.
Let
\begin{align*}
	\hat{g}_1(z)
	&= \min_{x':z\notin x'} \max_{x \supsetneq x'}  
	T^-({P}_{x \setminus x'}) 
	= \max_{x: z \in x} \min_{x' \subsetneq x} 
	T^-(P_{x \setminus x'}),\\
	\hat{g}_2(z) 
	&= -\min_{x':z\notin x'} \max_{x \supsetneq x'} -\E(\bar{P}_{x \setminus x'})
	= -\max_{x: z \in x} \min_{x' \subsetneq x} - \E(\bar{P}_{x \setminus x'}),
\end{align*}
where $\bar{P}_{i:j}$ is the empirical distribution of $L(g_1(z),y)$, $z\in \mathcal{Z}$.
\begin{proposition}\label{prop:criterion2}
Let $\hat{g}_1^-$, $\hat{g}_2^-$ as defined above.
A simultaneously optimal solution exists if and only if $\hat{g}_{1}^-=\hat{g}_{1;\mathcal{Z}\setminus \bar{x}}^-$ for all superlevel sets $\mathcal{Z}\setminus \bar{x}$, $\bar{x} \in \mathcal{X}$ assumed by $\hat{g}_2^-$. 
\end{proposition}
The reasoning behind this Proposition is analogous to the reasoning behind Proposition \ref{prop:criterion}.

\end{document}